\documentclass[11pt]{amsart} %

\usepackage{float}

\usepackage{epsfig}

\usepackage{array,amsmath, enumerate,  url, psfrag}
\usepackage{amssymb, amsaddr, fullpage}
\usepackage{graphicx, subfigure}
\usepackage[usenames,dvipsnames]{pstricks}
\usepackage{color}

\restylefloat{figure}
 \makeatother

\newtheorem{theorem}{Theorem}[section]
\newtheorem{lemma}[theorem]{Lemma}

\newtheorem{satz}{thm1}

\newtheorem{lem}[satz]{Lemma}

\theoremstyle{definition}
\newtheorem{defn}[satz]{Definition}

\title{Generalization of some results on list coloring and DP-coloring}
\author{Keaitsuda Maneeruk Nakprasit$^{1}$ \hskip 0.2in Kittikorn Nakprasit$^{1}$}

\address{
$^{1}$\small Department of Mathematics, Faculty of Science, Khon Kaen University, 40002, Thailand.}

\begin{document}

\maketitle

\begin{center}{\bf Abstract}\end{center}
\indent\indent
Let $G$ be a graph and let $f_i, i \in \{1,\ldots,s\},$
be a function   from $V(G)$ to the set of nonnegative integers.
In \cite{SiNa1}, the concept of DP-$F$-coloring,
a generalization of DP-coloring and variable degeneracy, was introduced.  
We use DP-$F$-coloring to define DPG-$[k,t]$-colorable graph    
and modify the proofs in \cite{DP3, SiNa2, Tho}  to obtain more results on
list coloring, DP-coloring, list-forested coloring, and  variable degeneracy.

\section{Introduction}
Every graph in this paper is finite, simple, and undirected.
We let $V(G)$ denote the vertex set and $E(G)$ denote the edge set of a graph $G.$ 
Let $d_G(v)$ denote the degree of a vertex $v$ in a graph $G.$
If no confusion arises, we simply use $d(v)$ instead of $d_G(v).$ 
Let $S$ be a subset of $V(G).$
A subgraph of $G$ induced by $S$ is denoted by $G[S].$ 
If a plane graph $G$ contains a cycle $C,$ we use
$int(C)$ (respectively, $ext(C)$) for the subgraph induced by vertices on $C$
and inside $C$ (respectively, outside $C$).  

Let $f$ be a function from $V(G)$ to the set of positive integers.
A graph $G$ is \emph{strictly} $f$-\emph{degenerate}
if every subgraph $G'$ has a vertex $v$ with $d_{G'}(v) < f(v).$
Equivalently, $G$ is strictly $f$-degenerate
if and only if vertices of $G$ can be ordered so that 
each vertex has less than $f(v)$ neighbors in the lower order.
Let $k$ be a positive integer.
A graph $G$ is \emph{strictly} $k$-\emph{degenerate} 
if and only if $G$ is strictly $f$-degenerate where $f(v)=k$
for each vertex $v.$
Thus a strictly $1$-degenerate graph is an edgeless graph
and a strictly $2$-degenerate graph is a forest.
Equivalently, $G$ is strictly $k$-degenerate
if and only if vertices of $G$ can be ordered so that
each vertex has less than $k$ neighbors in the lower order.

Let $f_i, i \in \{1,\ldots,s\},$ be a function
from $V(G)$ to the set of nonnegative integers.
An $(f_1,\ldots,f_s)$-\emph{partition} of a graph $G$
is a partition of $V(G)$ into $V_1,\ldots, V_s$ such that
an induced subgraph $G[V_i]$ is strictly $f_i$-degenerate
for each $i\in \{1,\ldots,s\}.$
A $(k_1,\ldots,k_s)$-\emph{partition} where $k_i$ is a constant for each $i\in \{1,\ldots,s\}$
is an  $(f_1,\ldots,f_s)$-partition such that $f_i(v)=k_i$ for each vertex $v.$
We say that $G$ is $(f_1,\ldots,f_s)$-\emph{partitionable} if
$G$ has  an $(f_1,\ldots,f_s)$-partition.
By Four Color Theorem \cite{4color},
every planar graph is $(1,1,1,1)$-partitionable.
On the other hand, Chartrand and Kronk \cite{ChKr}
constructed planar graphs which are not $(2,2)$-partitionable.
Even stronger, Wegner \cite{Wegner} showed that
there exists a planar graph which is not $(2,1,1)$-partitionable.
Thus it is of interest to find sufficient conditions for planar graphs
to be $(1,1,1,1)$-, $(2,1,1)$-, or $(2,2)$-partitionable.

Borodin and Ivanova  \cite{BoIv} obtained a sufficient condition that implies
$(1,1,1,1)$-, $(2,1,1)$-, or $(2,2)$-partitionability as follows.  

\begin{theorem} \label{main00} (Theorem 6 in \cite{BoIv})
Every planar graph without $4$-cycles adjacent to $3$-cycles is
$(f_1, \ldots, f_s)$-partitionable
if  $f_1(v)+\cdots+f_s(v)\geq 4$ for each vertex $v,$
and $f_i(v)\in \{0,1,2\}$ for each $v$ and $i.$
\end{theorem}

The \emph{vertex-arboricity} $va(G)$ of a graph $G$ is the minimum number of subsets in which  
$V(G)$ can be partitioned so that each subset induces a forest.
This concept was introduced by Chartrand, Kronk, and Wall \cite{ChKrWa} as \emph{point-arboricity}.  
They proved that $va(G)\leq3$ for every planar graph $G.$  
Later, Chartrand and Kronk \cite{ChKr}  proved that this bound is sharp by providing  
an example of a planar graph $G$ with $va(G)= 3.$  
It was shown that determining the vertex-arboricity of a graph
is NP-hard by Garey and Johnson \cite{NPhard}
and determining whether $va(G)\leq2$ is NP-complete for maximal planar graphs $G$
by Hakimi and Schmeichel \cite{NPcomplete}.
Raspaud and Wang \cite{RaWa} showed that
$va(G)\leq\lceil\frac{k+1}{2}\rceil$ for every $k$-degenerate graph $G$.
It was proved that every planar graph $G$ has $va(G)\leq2$
when $G$ is without $k$-cycles for $k\in\{3,4,5,6\}$ (Raspaud and Wang \cite{RaWa}),
without $7$-cycles (Huang, Shiu, and Wang \cite{HuShWa}),
without intersecting $3$-cycles (Chen, Raspaud, and Wang  \cite{ChRaWa}),  
without chordal $6$-cycles (Huang and  Wang \cite{HuWa}), or
without intersecting $5$-cycle (Cai, Wu, and Sun \cite{CaiWuSun}).

The concept of list coloring was independently introduced
by Vizing \cite{Vizing} and by Erd\H os, Rubin, and Taylor \cite{Erdos}.
A \emph{$k$-assignment} $L$ of a graph $G$ assigns a list $L(v)$ (a set of colors)
with $|L(v)|=k$ to each vertex $v$ of  $G$.
A graph $G$ is \emph{$L$-colorable}
if there is a proper coloring $c$ where $c(v)\in L(v)$ for each vertex $v.$
If $G$ is $L$-colorable for each $k$-assignment $L$,
then we say $G$ is \emph{$k$-choosable}.
The \emph{list chromatic number} of $G,$ denoted by $\chi_{l}(G),$
is the minimum number $k$ such that $G$ is $k$-choosable.

Borodin, Kostochka, and Toft \cite{BoKoTo}  
introduced list vertex arboricity
which is list version of vertex arboricity.  
We say that $G$ has an \emph{$L$-forested-coloring} $f$ for a set $L = \{L(v)|v \in V(G)\}$
if  one can choose $f(v) \in L(v)$ for each vertex $v$
so that a subgraph induced by vertices with the same color is a forest.
We say that $G$ is \emph{list vertex $k$-arborable} if $G$ has an $L$-forested-coloring for each
$k$-assignment $L.$  
The \emph{list vertex arboricity} $a_l(G)$ is defined to be  the minimum $k$ such that
$G$ is list vertex $k$-arborable.
Obviously, $a_l(G) \geq va(G)$ for every graph $G$.

It was proved that every planar graph $G$ is list vertex $2$-arborable  
when $G$ is without $k$-cycles for $k\in\{3,4,5,6\}$   (Xue and Wu \cite{XuWu}),
with no $3$-cycles at distance less than $2$ (Borodin and Ivanova \cite{BoIvC33}),  
or without $4$-cycles adjacent to $3$-cycles (Borodin and Ivanova \cite{BoIv}).  

Borodin, Kostochka, and Toft \cite{BoKoTo} observed that the notion of
 $(f_1,\ldots,f_s)$-partition can be applied to problems in list coloring
and list vertex arboricity.
Since $v$ cannot have less than zero neighbor, the condition 
that $f_i(v)=0$ is equivalent to $v$ cannot be colored by  $i.$
In other words, $i$ is not in the list of $v.$    
Thus the case of $f_i \in \{0,1\}$ corresponds to list coloring,
and one of  $f_i \in \{0,2\}$ corresponds to $L$-forested-coloring.  
Note that Theorem \ref{main00}  
implies that planar graphs without $3$-cycles adjacent to $4$-cycles
are $4$-choosable and list vertex $2$-arborable.  

Dvo\v{r}\'{a}k and Postle \cite{DP} introduced a generalization
of list coloring in which they called a \emph{correspondence coloring}.
Following Bernshteyn, Kostochka, and Pron \cite{BKP},  
we call it a \emph{DP-coloring}.

\begin{defn}\label{cover}
Let $L$ be an assignment of a graph $G.$
We call $H$ a \emph{cover} of $G$
if it satisfies all the followings:\\
(i) The vertex set of $H$ is $\bigcup_{u \in V(G)}(\{u\}\times L(u)) =
\{(u,c): u \in V(G), c \in L(u) \};$\\
(ii) $H[\{u\}\times L(u)]$ is a complete graph for each $u \in V(G);$\\
(iii) For each $uv \in E(G),$
the set $E_H(\{u\}\times L(u), \{v\}\times L(v))$ is a matching (may be empty);\\
(iv) If $uv \notin E(G),$ then no edges of $H$ connect  
$\{u\}\times L(u)$ and  $\{v\}\times L(v).$\\
Let $(G,H)$ denote a graph $G$ with a cover $H.$
\end{defn}

\begin{defn}\label{DP}
A \emph{representative set} of $(G,H)$
is a set of vertices of size $|V(G)|$
containing exactly one vertex from each $\{v\}\times L(v).$
A DP-coloring of $(G,H)$ is a representative set
$R$ that $H[R]$ has no edges.  
We say that a graph $G$ is \emph{DP-$k$-colorable} if $(G,H)$ has
a DP-coloring for  each cover $H$ of $G$  with a  $k$-assignment $L.$    
The \emph{DP-chromatic number} of $G,$ denoted by $\chi_{DP}(G),$
is the minimum number $k$ such that $G$ is DP-$k$-colorable.
\end{defn}

If we define edges on $H$ to match exactly the same colors in $L(u)$ and $L(v)$
for each $uv \in E(G),$
then $(G,H)$ has a DP-coloring if and only
if $G$ is $L$-colorable.  
Thus DP-coloring is a generalization of list coloring.  
Moreover, $\chi_{DP}(G) \geq \chi_l(G).$
For example, Alon and Tarsi \cite{Alon}
showed that every planar bipartite graph is $3$-choosable,
while Bernshteyn and Kostochka \cite{Differ} 
obtained a bipartite planar graph $G$ with $\chi_{DP}(G) = 4.$

Dvo\v{r}\'{a}k  and Postle \cite{DP} observed that
$\chi_{DP}(G) \leq 5$ for every planar graph $G.$  
This extends a seminal result by Thomassen \cite{Tho} on list colorings.
On the other hand, Voigt \cite{Vo1}  gave an example of a planar graph
which is not $4$-choosable (thus not DP-$4$-colorable).
Kim and Ozeki \cite{KimO} showed that planar graphs without $k$-cycles
are DP-$4$-colorable for each $k =3,4,5,6.$  
Kim and Yu  \cite{KimY} extended the result on $3$- and $4$-cycles
by showing that planar graphs without $3$-cycles adjacent to $4$-cycles  
are DP-$4$-colorable.

Later, the concept of DP-coloring and improper coloring  
is combined by allowing a representative set $R$
to yield $H[R]$ with edges but requiring $H[R]$ to satisfy
some degree conditions such as degeneracy \cite{SiNa1} or maximum degree \cite{SiNa2}.

\begin{defn}\label{DPrelax}
A \emph{DP-forested-coloring} of $(G,H)$ is a representative set $R$
such that $H[R]$ is a  forest.
We say that a graph $G$ is \emph{DP-vertex-$k$-arborable} if $(G,H)$ has
a DP-forested-coloring for each $k$-assignment $L$ and each cover $H$ of $G.$  
\end{defn}

If we define edges on $H$ to match exactly the same colors in $L(u)$ and $L(v)$
for each $uv \in E(G),$ then $(G,H)$ has a DP-forested-coloring
if and only if  $G$ has an $L$-forested-coloring.

From now on, we assume $G$ is a graph with a $k$-assignment of colors $L$
such that $\bigcup_{v\in V(G)} L(v) \subseteq \{1, \ldots, s\}$  
and $H$ is a cover of $G.$
Assume furthermore that $F = (f_1, \ldots, f_s)$ and $f_i,$ where $1 \leq i \leq s,$
is a function from $V(G)$ to the set of nonnegative integers.
The concept of DP-coloring is combined with $(f_1,\ldots,f_s)$-partition in \cite{SiNa1} as follows.

\begin{defn}
A \emph{DP-$F$-coloring} $R$ of $(G,H)$ is a representative set
which can be ordered so that each element $(v,i)$ in $R$
has less than $f_i(v)$ neighbors in the lower order.
Such order is called a \emph{strictly $F$-degenerate order}.
We say that $G$ is \emph{DP-$F$-colorable} if $(G,H)$ has a DP-$F$-coloring for every cover $H$.
\end{defn}

If we define edges on $H$ to match exactly the same colors
for each $uv \in E(G),$ then $G$ has an $(f_1,\ldots,f_s)$-partition
if and only if $(G,H)$ has a DP-$F$-coloring.
Thus an $(f_1,\ldots,f_s)$-partition is a special case
of a DP-$F$-coloring.
Observe that a DP-$F$-coloring where $f_i(v) \in \{0,1\}$ for each
$i$ and each vertex $v$ is equivalent to a DP-coloring.
Furthermore, a DP-$F$-coloring where $f_i(v)\in \{0,1,2\}$ for each
$i$ and each vertex $v$ is equivalent to a DP-forested-coloring.
We show in this work that  the condition $f_i(v) \in \{0,1\}$  
(DP-coloring) may be relaxed to $f_i(v) \in \{0,1,2\}$ to obtain a more general result.  
For conciseness, we define the following definition.

\begin{defn} \label{DPG}
Let $|f(v)|$  denote $f_1(v)+\cdots+f_s(v).$  
A graph $G$ is \emph{DPG-$[k,t]$-colorable} if $(G,H)$
has a DP-$F$-coloring for
every cover $H$ and $f$ such that  $|f(v)| \geq k$
and $f_i(v) \leq t$ for every vertex $v$
and every $i$ with $1 \leq i \leq s.$    
\end{defn}

\begin{lemma}
Let $C(i)$ denote the set of vertices colored $i$ in $G.$
If $G$ is DPG-$[k,2]$-colorable, then we have the followings:\\
(1) $G$ is DP-$k$-colorable and thus $k$-choosable.\\
(2) $G$ is DP-vertex-$\lceil{k/2} \rceil$-arborable.\\
(3) Let $2d > k.$
If $L$ is a $d$-assignment for $G$ where $d\leq k$
    and $1, 2, \ldots, 2d-k $ are colors,
then we can find an $L$-foreted-coloring such that
$C(i)$ is an independent set for each $i \in \{1, \ldots, 2d-k\}.$  
\end{lemma}
\begin{proof} Let $G$ be a DPG-$[k,2]$-colorable graph.\\
(1) Let $L$ be a $k$-assignment of $G.$
Define $f_i(v) =1$ if $i \in L(v),$ otherwise $f_i(v)=0.$
Note that $(G,H)$ has a DP-$k$-coloring if and only if $(G,H)$ has a DP-$F$-coloring.
Since $G$ is DPG-$[k,2]$-colorable, $(G,H)$ has a DP-$k$-coloring for every cover $H.$\\  
(2) Let $L$ be a $\lceil{k/2} \rceil$-assignment of $G.$
Define $f_i(v) =2$ if $i \in L(v),$ otherwise $f_i(v)=0.$
Note that $(G,H)$ has a DP-forested-coloring
if and only if $(G,H)$ has a DP-$F$-coloring.
Since $G$ is a DPG-$[k,2]$-colorable graph, $(G,H)$ has a DP-forested-coloring
for every cover $H$ and every $\lceil{k/2} \rceil$-assignment of $G.$\\
(3) Let $L$ be a $d$-assignment of $G.$
Define $f_i(v) =1$ when $i \in L(v)$ and $1 \leq i \leq 2d - k,$
$f_i(v) =2$ when $i \in L(v)$ and $i  \geq 2d - k+1,$
and $f_i(v)=0$ otherwise.  
Let edges on $H$ match exactly the same colors.  
Note that $G$ has an $L$-forested-coloring with $C(i)$
is an independent set for $1 \leq i \leq 2d - k$
if and only if $(G,H)$ has a DP-$F$-coloring.
Since $G$ is DPG-$[k,2]$-colorable, we have the desired result.\\
\end{proof}

We use the concept of DPG-$[k,2]$-colorable graph to generalize these three results on list coloring and DP-coloring.  

\begin{theorem} \label{main01} \cite{Tho}
Every planar graph is $5$-choosable.
\end{theorem}

\begin{theorem} \label{main02} \cite{SiNa2}
Let $\mathcal{A}$ be the family of planar graphs
without pairwise adjacent $3$-, $4$-, and $5$-cycles.
If $G \in \mathcal{A}$ contains a $3$-cycle $C,$
then each precoloring of $C$ can be extended to a DP-$4$-coloring of $G.$  
\end{theorem}

\begin{theorem} \label{main03} \cite{DP3}
Let $G$ be a planar graph without cycles of lengths $\{4,a, b, 9\}$
where $a$ and $b$ are distinct values from $\{6, 7, 8\}.$
Then $G$ is DP-$3$-colorable.
\end{theorem}

Using DPG-$[k,2]$-colorability,
we modify the proof of Theorems \ref{main01}, \ref{main02},
and \ref{main03} to obtain the following main results.

\begin{theorem}\label{main1}
Every planar graph $G$ is DPG-$[5,2]$-colorable.
In particular, we have the followings.\\
(1) $G$ is $5$-choosable \cite{Tho}. \\
(2) $G$ is $5$-DP-colorable \cite{DP}.\\
(3) If $L$ is a $4$-assignment of $G$ with colors $i,j,$ and $k,$
then $G$ has an $L$-forested-coloring with $C(i), C(j),$ and $C(k)$ are independent sets.\\
(4) If $L$ is a $3$-assignment of $G$ with a color $i,$
then $G$ has an $L$-forested-coloring with $C(i)$ is an independent set.\\
(5) $G$ is DP-vertex-$3$-arborable.\\
(6) $G$ is $(f_1,\ldots,f_s)$-partitionable if $|f(v)| \geq 5$ and
$f_i(v) \in \{0,1,2\}$  for every vertex $v$ and every $i$ with $1 \leq i \leq s.$
\end{theorem}

\begin{theorem}\label{main2}
Let $G \in \mathcal{A}$ contains a $3$-cycle $C_0.$  
Let $|f(v)| \geq k$ and $f_i(v) \leq 2$ for $1 \leq i \leq s.$
Then every DP-$F$-coloring on $C_0$  can be  
extended to a DP-$F$-coloring on $G.$  
In particular, we have the followings.\\
(1) $G$ is DP-$4$-colorable \cite{SiNa2}.\\
(2) If $L$ is a $3$-assignment of $G$ with colors $i$ and $j,$
then $G$ has an $L$-forested-coloring
with $C(i)$ and $C(j)$ are independent sets.\\
(3) $G$ is DP-vertex-$2$-arborable.\\
(4) $G$ is $(f_1,\ldots,f_s)$-partitionable if $|f(v)| \geq 4$ and
$f_i(v) \in \{0,1,2\}$  for every vertex $v$ and every $i$ with $1 \leq i \leq s.$\\  
Note that (1), (2), and (3) still hold even when $G$ has a corresponding precoloring on $C_0.$
\end{theorem}

\begin{theorem}\label{main3}
Let $G$ be a planar graph without cycles of lengths $\{4,a, b, 9\}$
where $a$ and $b$ are distinct values from $\{6, 7, 8\}.$
Then $G$ is DPG-$[3,2]$-colorable.
In particular, we have the followings.\\
(1) $G$ is DP-$3$-colorable \cite{DP3}.\\
(2) $G$ is DP-vertex-$2$-arborable.\\
(3) If $L$ is a $2$-assignment of $G$ with a color $i,$
then $G$ has an $L$-forested-coloring with $C(i)$ is an independent set.\\
(4) $G$ is $(f_1,\ldots,f_s)$-partitionable if $|f(v)| \geq 3$ and
$f_i(v) \in \{0,1,2\}$  for every vertex $v$ and every $i$ with $1 \leq i \leq s.$
 \end{theorem}

\section{Helpful Tools}

Some definitions and lemmas which are used to prove the main results
are presented in this section.
Since we focus on DP-$[k,2]$-colorability,  
we assume from now on that
$f_i(v) \in \{0,1,2\}$  for every vertex $v$ and every $i$ with $1 \leq i \leq s.$
Furthermore, a DP-$F$-precoloring on a subgraph $G'$ is assumed to be a DP-$F$-coloring
restrict on $(G', H')$ where $H'$ is a cover $H$ restrict to $G'.$
\begin{defn}\label{residual}  
Let $R'$ be a DP-$F$-precoloring on an induced subgraph $G'$ of $G.$
The \emph{residual function} $f^*=(f^*_1,\ldots,f^*_s)$
for $G - G'$ is defined by    
$$f_i^*(v)=\max \{0, f_i(v)-|\{(x,j) \in R' : (v,i)(x,j)\in E(H)\}|\}$$
for each $v \in V(G)-V(G').$
\end{defn}

For conciseness, we simply say $R_2$ is a DP-$F^*$-coloring of $G - G'$
instead of that of $(G- G', H- H').$
From the above definition, we have the following fact.

\begin{lem}\label{extend}
Let $R'$ be a DP-$F$-precoloring of an induced subgraph $G'$  of $G$
and let $F^*=(f^*_1,\ldots,f^*_s)$ be a residual function of $G - G'.$
If $G-G'$ has a DP-$F^*$-coloring, then $(G,H)$ has a DP-$F$-coloring.
\end{lem}
\begin{proof}
Let $R_1$ be a DP-$F$-precoloring of $G'$ with a strictly $F$-degenerate order
$S_1$ and $R_2$ be a DP-$F^*$-coloring of $G-G'$
with a strictly $F^*$-degenerate order $S_2.$ 
Then $R_1 \cup R_2$ is a representative set of $(G,H).$
We claim that the order $S$ obtained
from $S_1$ followed by $S_2$ is a strictly $F$-degenerate order  
of $R_1 \cup R_2.$ Consequently,
$R_1 \cup R_2$ is a DP-$F$-coloring of $(G,H).$
For $(v,i) \in R_1,$ the neighbors in the lower order of $S$ and that of $S_1$
are the same.
By the construction of $S_1,$
$(v,i)$ has less than $f_i(v)$ neighbors in the lower order of $S.$
Consider $(v,i) \in R_2.$
Suppose $(v,i)$ has $d$ neighbors in $R_1.$
Note that $f_i^*(v) \geq 1,$ otherwise $(v,i)$ cannot be chosen in $R_2.$
It follows that  $f^*_i(v) = f_i(v)-d$ by the definition of $f^*_i.$
Since $(v,i)$ has less than $f_i^*(v)$ neighbors in $R_2$ in the lower order of $S,$
$(v,i)$ has less than $f_i^*(v)+d = f_i(v)$ neighbors  
in the lower order of $S.$
Thus $S$ is a strictly $F$-degenerate order.
\end{proof}

Similarly, a partial DP-$F$-coloring $R'$ with a strictly $F$-degenerate order $S$
can be extended by a greedy coloring on a vertex $v$ with $|f^*(v)|\geq 1.$  
We add $(v,i)$ with $f_i(v) \geq 1$ to $R'.$
It can be seen that $S$ followed by $(v,i)$ is a strictly $F$-degenerate order.

The term \emph{minimal counterexample} is used 
for $(G,H)$ that is a counterexample and $|V(G)|$ is minimized.

\begin{lemma} \label{min1}
If $(G, H)$ is a minimal counterexample to Theorem \ref{main3},
then every vertex has degree at least $3.$
\end{lemma}
\begin{proof}
Suppose to the contrary that a vertex $v$ has degree at most $2.$
By minimality, $G-v$ has a DP-$F$-coloring.
Now, $|f^*(v)|\geq |f(v)| - d(v) \geq 3 - 2 =1.$
Thus we can apply a greedy coloring to $v$ to complete the coloring.
\end{proof}

With a similar proof, one obtain the following lemma.

\begin{lemma} \label{min2}
If $(G,H)$ and a precolored $3$-cycle $C_0$
is a minimal counterexample to Theorem \ref{main2},
then every vertex not on $C_0$ has degree at least $4.$
\end{lemma}

\begin{lemma} \label{front}
Let $G$ be a graph containing a subgraph $K$ with the following property:  
if  $H$ is a cover of $G$ and $f$ has $f(v)| \geq k$
for every vertex $v,$ then each DP-$F$-coloring of $K$
can be extended to that of $(G,H).$
Suppose $R_1$ is a DP-$F$-coloring of $K.$
Then there exists a DP-$F$-coloring of $(G,H)$ with
a strictly $F$-degenerate $S$ such that
the $|R_1|$ lowest-ordered elements are in $R_1.$
\end{lemma}

\begin{proof}
Let $R_1$ be a DP-$F$-coloring of $K$ with a strictly $F$-degenerate order $S_1.$
By renaming the colors, we assume that $S_1$ has the order $(v_1,1),\ldots,(v_t,1).$
Let $H'$ be a cover of $G$ obtained from $H$ by modifying matchings
between colors in $R_1$ so that $R_1$ is independent.

Let $f'$ be obtained from $f$ by defining $f'_i(v_1)=\cdots =f'_i(v_t)=1$ if
$1 \leq i \leq k,$ otherwise $f'_i(v_1)=\cdots =f'_i(v_t)=0.$
Note that $|f'(v)| \geq k$ and $f'_i(v) \in \{0,1, 2\}$ for every vertex $v$
and every $i$ with $1 \leq i \leq k.$
By condition of $G$ and $K,$ $(G,H')$ has a DP-$f'$-coloring $R$ with
a strictly $f'$-degenerate order $S'.$
Let $S$ be obtained from $S'$ by moving $(v_1,1),\ldots,(v_t,1)$ to be in
the lowest order.
We claim that $R$ is a DP-$F$-coloring with
a strictly $F$-degenerate order $S.$

It is obvious that $R$ is a representative set of $(G,H)$
and $(v_1,1),\ldots,(v_t,1)$ are the lowest elements of $S.$
It remains to show that $S$ is a strictly $F$-degenerate order.
Consider $(u,i) \in R.$
If $(u,i) \in R_1,$ then it has less than $f_i(u)$ neighbors
in the lower order of $S_1$ by the construction.
Since the neighbors in the lower order of $S_1$ and that of $S$ are the same,
$(u, i)$ has less than $f_i(u)$ neighbors in the lower order of $S.$

Assume that $(u,i) \notin R_1.$
Suppose to the contrary that  $(u,i)$ has at least $f_i(u)$ neighbors
in the lower order of $S.$
Since $S'$ is a strictly $f'$-degenerate order,
$(u,i)$ has less than $f'_i(u)=f_i(u)$ neighbors
in the lower order of $S'.$
Then an additional neighbor in the lower order of $S,$ say $(v,1),$
is in $R_1$ by the construction of $S.$
Moreover, the order of $(u,i)$ in $S'$ is lower than that of $(v,1).$
It follows that $(v,1)$ has at least $f'_1(v) =1$ neighbor in the lower order of
a strictly $f'$-degenerate order $S',$ a contradiction.
It follows that $(u, i)$ has less than $f_i(u)$ neighbors
in the lower order of $S.$
Thus $S$ is a strictly $F$-degenerate order and this completes the proof.
\end{proof}

Note that Lemma \ref{front} holds regardless of an upper bound on $f_i(v).$

\begin{lemma} \label{sep}
Let $(G,H)$ be a minimal counterexample to Theorem \ref{main2}
with a DP-$F$-precoloring of $3$-cycle $C_0. $
Then $G$ has no separating $3$-cycles.
\end{lemma}
\begin{proof}
Suppose to the contrary that $G$ has a separating $3$-cycle $C.$
By symmetry, we assume $C_0 \subseteq ext(C).$
By minimality, a DP-$F$-coloring on $C_0$
can be extended to a coloring $R_1$ on $ext(C).$
Let $S_1$ be a strictly $F$-degenerate order of $R_1.$
Let $V(C) =\{x,y,z\}$ and $(x,1), (y,1), (z,1) \in R_1.$
By minimality, $int(C)$ has a DP-$F$-coloring $R_2$
including $(x,1),(y,1),(z,1).$
By Lemma \ref{front}, $R_2$ has a strictly $F$-degenerate order $S_2$ such that
$(x,1),(y,1),(z,1)$ are the lowest order elements.

It is obvious that $R_1 \cup R_2$ is a representative set of $(G,H).$
Let $S'_2$ be obtained from $S_2$ by deleting $(x,1),(y,1),(z,1).$
We claim that $S$ obtained from $S_1$ followed by $S'_2$
is a strictly $F$-degenerate order.
If $(u,i) \in R_1,$  then the neighbors of $(u,i)$ in the lower order of $S$
are the same as that of $S_1$ by the construction of $S.$
It follows from $S_1$ is a strictly $F$-degenerate that
$(u,i)$ has less than $f_i(u)$ neighbors in the lower order of $S.$
Note that this case also includes $(u,i)$ is $(x,1),(y,1)$ or $(z,1).$

Consider $(u,i) \in R_2 - R_1.$
Then $(u,i)$ has less than $f_i(v)$ neighbors in the lower order of $S_2.$
It follows that $(u,i)$ has less than $f_i(v)$ neighbors that are in $R_2$
and in the lower order of $S.$
Since $(u,i)$ is not adjacent to any elements in $R_1 - \{(x,1),(y,1),(z,1)\},$
all neighbors of $(u,i)$ are in $R_2.$
Consequently, $(u,i)$ has less than $f_i(v)$ neighbors  in the lower order of $S.$
Thus $R_1 \cup R_2$ is a DP-$F$-coloring of $(G,H),$ a contradiction.
\end{proof}

\begin{lemma}\label{mainlemma}
Let $k\geq 3$ and $K \subseteq G$ with $V(K)=\{v_1,\ldots, v_m\}$
such that the followings hold. \\
(i)  $k- (d_G(v_1)- d_K(v_1)) \geq 3.$\\
(ii) $d_G(v_m) \leq k$ and neighbors of $v_m$ in $K$ are exactly $v_1$ and $v_{m-1}.$\\
(iii) For $2 \leq i \leq m-1,$ $v_i$ has at most $k-1$ neighbors in $G[\{v_1,\ldots,v_{i-1}\}]$ $\cup (G - K).$\\
If $|f(v)| \geq k$ for every vertex $v,$ then a
DP-$F$-precoloring of $G - K$ can be extended to that of $G.$
\end{lemma}

\begin{proof}
Let $R_0$ be a DP-$F$-coloring on $G-K.$
From Condition (i),  $|f^*(v_1)| \geq |f(v)| - (d_G(v_1)- d_K(v_1))$
$\geq k - (d_G(v_1)- d_K(v_1))  \geq 3.$
From Condition (ii), $|f^*(v_m)| \geq |f(v_m)| - (k-2) \geq 2.$
We consider only the case $|f^*(v_m)| =2 $ since
a strictly $F^*$-degenerate order of  $R_2$  is also
a strictly $g$-degenerate if $g_i(v) \geq f^*_i(v)$ for every vertex $v$ and $i$ such that $1 \leq i \leq s.$
By renaming the colors, we assume that $(v_m,j)$ and $(v_i,j),$
where $i=1$ and $m-1,$ are adjacent for each $j.$
Since $|f^*(v_1)|\geq 3,$ we may assume further that $f^*_1(v_1)>f^*_1(v_m).$
By Lemma \ref{extend}, it suffices to show that $K$ has a DP-$F^*$-coloring.
Consider two cases.\\

\textbf{Case 1:} $\boldsymbol{f^*_1(v_m)=0.}$\\
Choose $(v_1,1)$ in a coloring. Observe that $|f^*(v_m)|$ remains the same.
Apply greedy coloring to $v_2,\ldots, v_{m-1},$ respectively.
At this stage $|f^*(v_m)| \geq 1,$
thus we can use greedy coloring to $v_m$ to complete a DP-$F^*$-coloring.\\

\textbf{Case 2:}  $\boldsymbol{f^*_1(v_m) \geq 1.}$\\
Recall that we consider only $f_i(v) \in \{0,1,2\}$ for each vertex $v$
and every $i$ such that $1 \leq i \leq s.$
It follows that $f^*_1(v_1) =2$ and  $f^*_1(v_m) = 1.$
Since $|f^*(v_m)| =2,$ we  assume that $f^*_2(v_m)=1.$
Choose $(v_1,1)$ in a coloring.
We can apply greedy coloring to $v_2,\ldots, v_{m-1},$ respectively.
By Condition (iii), $K- v_m$ has a DP-$F^*$-coloring, say $R.$
By Condition (ii), $(v_m,2)$ has exactly two neighbors in $H$ restrict to $K.$

If $(v_{m-1},2)$ is not in $R,$ then $(v_m,2)$ has no neighbors in $R.$
Thus we can add $(v_m,2)$ to $R$to complete a DP-$F^*$-coloring.
Assume otherwise that $(v_{m-1},2) \in R.$
Let $(v_i, j_i) \in R$ for $2 \leq i \leq m-2.$
By greedy coloring,
we have a strictly $F^*$-degenerate order $S_1=(v_1,1),$ $(v_2,j_2),$ $\ldots,$ $(v_{m-2},j_{m-2}),$ $(v_{m-1}, 2).$

We claim that the order $S$ constructed
from $(v_m,1)$ followed by $S$ is a strictly $F^*$-degnerate order.  
It is obvious that $(v_1,1)$ has less than $f^*_1(v_1) =2$ neighbors
in the lower order.
Consider $(v_i, j_i)$ where $2 \leq i \leq m-2.$
Since $(v_i, j_i)$ is not adjacent to $(v_m,1)$ by Condition (ii),
$(v_i, j_i)$ has less than $f^*_{j_i}(v_i)$ neighbors
in the lower order of $S.$
Since  $(v_{m-1},2)$ is not adjacent to $(v_m,1),$
the element $(v_{m-1},2)$ has less than $f^*_2(v_{m-1}).$
It is obvious that the set of elements in the order of $S$ is
a representative set of $K.$
Thus $K$ has a DP-$F^*$-coloring.
This completes the proof.
\end{proof}

\section{Proofs of Main Results}
\noindent \textbf{Proof of Theorem} \ref{main1}\textbf{.}
The outline of the proof is similar to that in \cite{Tho}
with additional details on DP-$F$-coloring.
We begin by adding new edges in a plane graph
until we obtain a plane graph $G$
such that every bounded face is a triangle.
Let $|f(v)| \geq 5$ for each vertex $v.$
Let a cycle $C= v_1\ldots v_p$ be the boundary of the unbounded face.
Using induction on $|V(G)|,$
we prove the stronger result that a DP-$F$-coloring can be achieved even when
$v_1$ and $v_p$ have been precolored and $|f(v_i)| \geq 3$ for $2\leq i \leq p-1.$
Let $\{(v_1,a),(v_p,b)\}$ be a DP-$F$-precoloring.
If $|V(G)|=3,$ the vertex $v_2$ can be greedily colored.
Consider $|V(G)| \geq 4$ for the induction step. \\

\textbf{Case 1: } $\boldsymbol{C}$ \textbf{has a chord}
$\boldsymbol{v_iv_j}$ \textbf{with}
$\boldsymbol{1 \leq i \leq j-2 \leq p-1.}$

Let $C_1$ be the cycle $v_1v_2\ldots v_i$ $v_j v_{j+1}\ldots v_p$
and let $C_2$ be the cycle $v_jv_iv_{i+1}\ldots v_{j-1}.$
Let $G_1 = int(C_1)$ and let $G_2 =int(C_2).$
By induction hypothesis and Lemma \ref{front},
$G_1$ has a DP-$F$-coloring $R_1$
with a strictly $F$-degenerate order $S_1$
such that two lowest elements are $(v_i,1)$ and $(v_j,1).$  
It follows from Lemma \ref{front} that $G_2$ has
a DP-$F$-coloring $R_2$ with a strictly $F$-degenerate order $S_2$
with two lowest elements $(v_i,1)$ and $(v_j,1).$  
Let $S'_2$ be an order obtained from $S_2$ by removing  $(v_i,1)$ and $(v_j,1).$  
It can be shown as in the proof of Lemma \ref{sep} that
$R_1 \cup R_2$ is a representative set with a strictly $F$-degenerate order
obtained from $S_1$ followed by $S'_2.$

\textbf{Case 2:} $\boldsymbol{C}$ \textbf{has no chords.}\\
Let $v_1, u_1,u_2,\ldots,u_m,v_3$ be the neighbors of $v_2$ in order.
Let $U$ denote $\{u_1,\ldots,u_m\}$
and $G'$ denote $G- \{v_2\}.$
Using a DP-$F$-coloring on $v_1$ and $v_p,$
we have $|f^*(v_2)| \geq |f(v_2)|-1=2$ for $p \geq 4$
and $|f^*(v_2)| \geq |f(v_2)|-2=1$ for $p = 3.$
By renaming the  colors, we assume furthermore
that $(v_2,i)$ is adjacent to $(u,i)$
for each $u\in U \cup \{v_3\}$ and $ 1 \leq i \leq s.$
Let $f^*_1(v_2)= \max\{f^*_1(v_2),\ldots,f^*_s(v_2)\}.$

\textbf{Case 2.1:}  $\boldsymbol{p =3}$ \textbf{ or }
$\boldsymbol{f^*_1(v_2)\geq 2.}$\\
We choose $(v_2,1)$ in a DP-$F$-coloring.
Let $f'$ be obtained from $f$ by letting $f'_1(u) = 0$ for each $u \in U.$  
Since $f_1(u) \leq 2,$    we have $|f'(u)| \geq 3$ for each $u \in U.$  
By induction hypothesis and Lemma \ref{front},
$G'$ has a DP-$f'$-coloring $R'$ with a strictly $f'$-degenerate order $S'$
such that $(v_1,a)$ and $(v_{p=3},b)$ are the first two elements.

Suppose $p=3.$
Let $S$ be obtained from $S'$ by inserting $(v_2,1)$ as the third element.
Since $f^*_1(v_2) \geq 1$ when we have a precoloring $\{(v_1,a),(v_p,b)\},$
the element $(v_2,1)$ can be chosen by a greedy coloring.

Note that the only neighbors of $v_2$ are $v_1,v_3,$ and vertices in $U.$  
If $u \in U,$ then $(u,1)$ is not in $R'$ since $f'_1(u)=0.$
Thus $(v,c)$ where $v \notin U \cup \{v_1, v_3\}$ has
less than $f'_c(v) = f_c(v)$ neighbors in the lower order of $S.$  
Thus $S$ is a strictly $F$-degenerate order of $R' \cup \{(v_2,1)\}.$
It is obvious that $R' \cup \{(v_2,1)\}$  is a representative set
and thus a DP-$F$-coloring.  

Suppose $p=4$ and  $f^*_1(v_2)\geq 2.$
After a coloring on $G',$  we have $f^*_1(v_2)\geq 2 -1$ since
the only possible neighbor of $(v_2,1)$ other than $(v_1,a)$
in the coloring $R_1$ is $(v_3,1).$
Thus a greedy coloring can be applied to $v_2.$\\

\textbf{Case 2.2:}   $\boldsymbol{p \geq 4}$ \textbf{ and }
 $\boldsymbol{f^*_1(v_2)=1.}$\\
Since $|f^*(v_2)|\geq 2$ and by symmetry,
we assume $f^*_2(v_2) =1.$
Define $g_i(v_2) = f^*_i(v_2).$
Let $f'$ be obtained from $f$ by letting $f'_1(u)=\max \{0, f_1(u)-1\},$
$f'_2(u)=\max \{0, f_2(u)-1\}.$  
Observe that  $|f'(u)| \geq 3$ for each $u \in U.$
By induction hypothesis,
$G'$ has a DP-$f'$-coloring $R'$ (thus a DP-$F$-coloring).
It follows from Lemma \ref{front} that $R'$
has a strictly $f'$-degenerate order $S'$
with $(v_1,a)$ and $(v_p,b)$ are the two lowest ordered elements.

Let $t=1$ if  $(v_3,1)$ is not in $R',$ otherwise let $t=2.$
It is obvious that $R= R' \cup \{(v_2,t)\}$ is a representative set.
Let $S$ be an order obtained from inserting $(v_2,t)$
as the third element into $S'.$
We claim that $S$ is a strictly $F$-degenerate order of $R.$

Consider $(v_2,t).$ Since $p \geq 4,$ $(v_2,t)$ is not adjacent to $(v_p,b).$
If $t=a,$ then $f_t(v_2)= g_t(v_2) + 1 = 2,$
otherwise, $f_t(v_2)= g_t(v_2) =1.$
In both cases, $(v_2,t)$ has less than $f_t(v_2)$ neighbors
in the lower order of $S.$

Consider $(v, c)$ in $R$ where $v \notin \{v_1,v_2,v_p\}.$  
We have $(v,c)$ has less than $f'_c(v)$ neighbors other than $(v_2,t)$
in the lower order of $S$ by the construction of $S.$
If $(v,c)$ is adjacent to $(v_2,t),$
then $v \in U$ and $c=t.$
Consequently, $f_c(v)=f'_c(v)+ 1.$
If $(v,c)$ is not adjacent to $(v_2,t),$ then $f_c(v)\geq f'_c(v).$
In both cases, $(v,c)$ has less than $f_c(v)$ neighbors in the lower order of $S.$
Thus $S$ is a strictly $F$-degenerate of $R.$
This completes the proof.

\noindent \textbf{Modification of the Proof of Theorem} \ref{main2}\textbf{.}

For the proof of Theorem \ref{main2}, each configurations
that are forbidden to be contained in a minimal counterexample
are obtained from the fact that (i) $G \in \mathcal{A},$
(ii) $G$ has no separating $3$-cycles (Lemma \ref{sep})  
and the following lemma.
\begin{lemma}\label{degreemain}
Let $|f(v)| \geq 4$ for each vertex $v.$
Let $C$ be a cycle $x_1 \ldots x_m$ with $V(C)\cap V(C_0)= \emptyset$  
where $C_0$ is a precolored $3$-cycle.
Let $C(l_1,\ldots,l_k)$ be obtained from a cycle $C$ with  
$k-1$ internal chords sharing a common endpoint $x_1.$  
Suppose $K=G[C]$ contains $C(l_1,\ldots,l_k)$
where $x_2$ or $x_m$ is not the endpoint of any chord in $C.$
If $d_G(x_1) \leq k+2$ and $d_K(x_1) = k+1,$ then there exists $i \in \{2,3,\ldots,m\}$ such that $d(x_i) \geq 5.$  
\end{lemma}
One can see that Lemma \ref{degreemain} is immediate from Lemma \ref{mainlemma}
by assuming an order $x_1,\ldots,x_m$ with $x_m$ is not endpoint of any chord.
Thus all forbidden configurations required as in the proof of Theorem \ref{main02} in \cite{SiNa2} are obtained.
Using Lemma \ref{min2} about vertex degrees and the discharging method as in \cite{SiNa2},
one can complete the proof. \\
\begin{figure*}
	\begin{center}
		\includegraphics[width=0.7\textwidth]{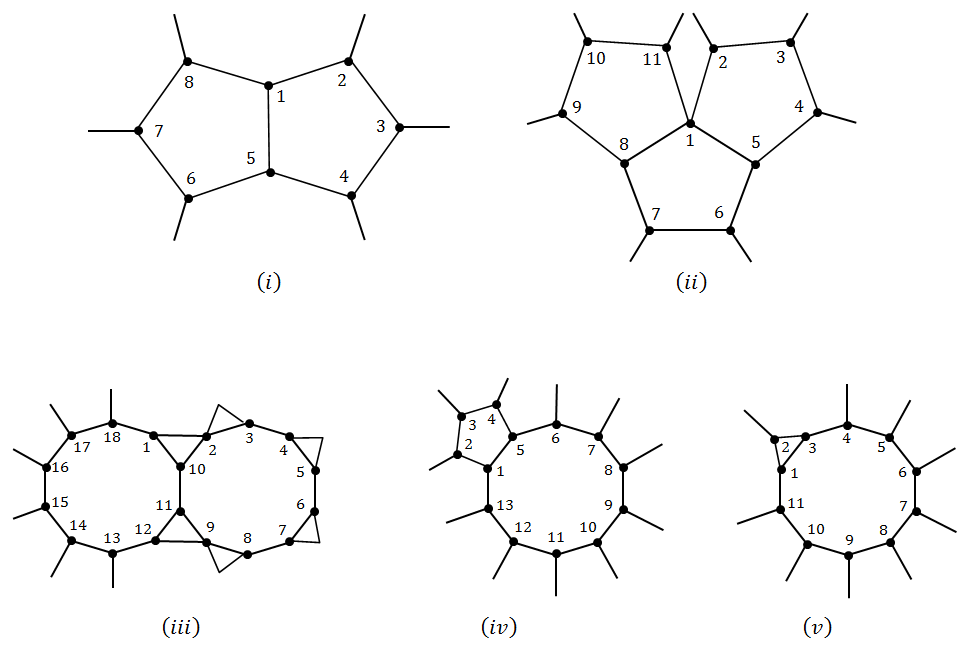}
		\caption{Forbidden configurations in Theorem \ref{main2}}
		\label{Pic1}       
	\end{center}
\end{figure*}

\noindent \textbf{Modification of the Proof of Theorem} \ref{main3}\textbf{.}
All five forbidden configurations of minimal counterexample to
Theorem \ref{main3} (as in Lemma 2.3 of \cite{DP3}) are in (See Fig. 1).
Consider a subgraph $K$ induced by the labeled vertices 
and order the vertices according to labels. 
Note that all labeled vertices are different 
to avoid creating cycles of forbidden lengths. 
It can be proved by Lemma \ref{mainlemma} that 
DP-$F$-precoloring of $G - K$ can be extended to that of $G.$ 
Thus a minimal counterexample cannot contains configurations in Fig. 1.
Using Lemma \ref{min1} about vertex degrees and
the discharging method as in \cite{DP3},
one can complete the proof.

\noindent\textbf{Acknowledgment}
We would like to thank Tao Wang for pointing out a few gaps of proofs and
giving valuable suggestions for  earlier versions of manuscript.

\end{document}